\pdfoutput=1
\RequirePackage{ifpdf}
\ifpdf 
\documentclass[pdftex]{sigma}
\else
\documentclass{sigma}
\fi

\numberwithin{equation}{section}

\newtheorem{Theorem}{Theorem}[section]
\newtheorem{Corollary}[Theorem]{Corollary}
\newtheorem{Lemma}[Theorem]{Lemma}
\newtheorem{Proposition}[Theorem]{Proposition}
 { \theoremstyle{definition}
\newtheorem{Remark}[Theorem]{Remark} }

\begin{document}

\newcommand{\arXivNumber}{2105.08641}

\renewcommand{\thefootnote}{}

\renewcommand{\PaperNumber}{077}

\FirstPageHeading

 \ShortArticleName{Second-Order Differential Operators in the Limit Circle Case}

\ArticleName{Second-Order Differential Operators\\ in the Limit Circle Case\footnote{This paper is a~contribution to the Special Issue on Mathematics of Integrable Systems: Classical and Quantum in honor of Leon Takhtajan.

The full collection is available at \href{https://www.emis.de/journals/SIGMA/Takhtajan.html}{https://www.emis.de/journals/SIGMA/Takhtajan.html}}}

\Author{Dmitri R.~YAFAEV~$^{\rm abc}$}

\AuthorNameForHeading{D.R.~Yafaev}

\Address{$^{\rm a)}$~Universit\'e de Rennes, CNRS, IRMAR-UMR 6625, F-35000 Rennes, France}
\EmailD{\href{mailto:yafaev@univ-rennes1.fr}{yafaev@univ-rennes1.fr}}
\Address{$^{\rm b)}$~St.~Petersburg University, 7/9 Universitetskaya Emb., St.~Petersburg, 199034, Russia}
\Address{$^{\rm c)}$~Sirius University of Science and Technology, 1 Olympiysky Ave., Sochi, 354340, Russia}

\ArticleDates{Received May 20, 2021, in final form August 14, 2021; Published online August 16, 2021}

\Abstract{We consider symmetric second-order differential operators with real coefficients such that the corresponding differential equation is in the limit circle case at infinity. Our goal is to construct the theory of self-adjoint realizations of such operators by an analogy with the case of Jacobi operators. We introduce a new object, the quasiresolvent of the maximal operator, and use it to obtain a very explicit formula for the resolvents of all self-adjoint realizations. In particular, this yields a simple representation for the Cauchy--Stieltjes transforms of the spectral measures playing the role of the classical Nevanlinna formula in the theory of Jacobi operators.}

\Keywords{second-order differential equations; minimal and maximal differential operators; self-adjoint extensions; quasiresolvents; resolvents}

\Classification{33C45; 39A70; 47A40; 47B39}

\begin{flushright}
\begin{minipage}{65mm}
\it Dedicated to Leon Takhtajan\\ on the occasion of his 70th anniversary
\end{minipage}
\end{flushright}

\renewcommand{\thefootnote}{\arabic{footnote}}
\setcounter{footnote}{0}

\section{Introduction}

\subsection{Setting the problem}

It is a common wisdom that spectral properties of second-order differential operators and Jacobi operators are in many respects similar. Recently this analogy was used in articles \cite{Y/LD, JLR} to recover known and obtain some new results for
Jacobi operators with coefficients stabilizing at infinity. In this paper we move in the opposite direction and study
differential operators in the limit circle case relying on an analogy with similar problems for Jacobi operators.
Actually, we follow rather closely an approach developed for Jacobi operators in \cite{Jacobi-LC}.

We consider second-order differential operators ${\mathcal A}$ defined by the formula
 \begin{equation}
( {\mathcal A} u)(x)=- \big(p(x) u' (x)\big)' + q(x) u(x) ,\qquad x\in {\mathbb R}_{+},
 \label{eq:ZP+}\end{equation}
 and acting in
 the space $L^2 ({\mathbb R}_{+})$. The scalar product in this space is denoted $\langle \cdot, \cdot\rangle$; $I$ is the identity operator. We always suppose that functions $p(x)$ and $q(x)$ are real. Then the opera\-tor~${\mathcal A}$ defined on the set $ C_{0} ^\infty ({\mathbb R}_{+}) $ is symmetric, but to make it self-adjoint, one has to add boundary conditions at $x=0$ and, eventually, for $x\to\infty$. We suppose that conditions at these two points are separated. The boundary condition at the point $x=0$ looks as
 \begin{equation} \label{eq:BCz}
u'(0)= \alpha u(0), \qquad\mbox{where}\quad \alpha =\bar{\alpha}.
 \end{equation}
 The value $\alpha=\infty$ is not excluded. In this case \eqref{eq:BCz} should be understood as the equality $u(0)=0$. We always require condition \eqref{eq:BCz}, fix $\alpha$ and do not keep track of $\alpha$ in notation.

 Our objective is to study a singular case, where all solutions $u$ of the equation ${\mathcal A}u=z u $ for $z\in{\mathbb C}$ are in $L^2 ({\mathbb R}_{+})$. This instance is known as the limit circle (LC) case. In this case the operator ${\mathcal A}$ with boundary condition~\eqref{eq:BCz} has a one parameter family of self-adjoint realizations distinguished by some conditions for $x\to\infty$. Their description can be performed in various terms. We here adopt an approach similar to the one used for Jacobi operators as presented in the book \cite[Section~16.3]{Schm} or in the survey \cite[Section~2]{Simon}.

 \subsection{Structure of the paper}\label{section1.2}

 In Sections~\ref{section2.1} and~\ref{section2.2}, we collect standard information about differential equations of second-order and realizations of differential operators ${\mathcal A}$ in the space $L^2 ({\mathbb R}_{+})$. We first define symmetric operators $A_{\min}$ with minimal domains
 $\mathcal{D}(A_{\min})$. Their self-adjoint extensions $A$ satisfy the condition
 \begin{equation}
A_{\min}\subset A =A^*\subset A_{\min}^*=: A_{\max}.
 \label{eq:Neum3}\end{equation}
In the LC case the operators $A_{\max}$ are not symmetric. In Section~\ref{section2.3}, we recall the traditional procedure of constructing self-adjoint extensions of the operator $A_{\min}$ in terms of some boundary conditions for $x\to\infty$. Then we suggest in Section~\ref{section2.4} an alternative approach to this problem, where self-adjoint extensions $A_{t}$, $t\in {\mathbb R}\cup\{ \infty\}$, of $A_{\min}$ are defined in a way analogous to the case of the Jacobi operators.

Our main result, an explicit formula for the resolvents $R_t (z)= (A_t-zI)^{-1}$, is obtained in Section~\ref{section3.2}, Theorem~\ref{RES}. Previously, we construct in Section~\ref{section3.1} (see Theorem~\ref{res}) an operator~${\mathcal R}(z)$ playing, in some sense, the role of the resolvent of the maximal operator $A_{\max}$. The operator~${\mathcal R}(z)$, we call it the {\it quasiresolvent}, is the key element of our construction.
 Note that the operator valued function~${\mathcal R}(z)$ depends analytically on $z\in{\mathbb C}$. Then, using the operator~${\mathcal R}(z)$, we prove Theorem~\ref{RES}. This also yields a representation (see Section~\ref{section3.3}) for spectral families~$E_{t}(\lambda)$ of~$A_{t}$ which is a modification of the Nevalinna formula in the theory of Jacobi operators; see the original paper \cite{Nevan} or \cite{Schm, Simon}.

\section{Differential equations and associated operators}\label{section2}

 We refer to the books \cite[Section~17]{Nai} and \cite[Section~X.1]{RS} for necessary background information on the theory of symmetric differential operators. A lot of relevant results can also be found in the encyclopedic book \cite{Zettl};
 see, in particular, Chapter~10.

\subsection{Limit point versus limit circle}\label{section2.1}

Let us consider a second-order differential equation
\begin{equation}
- (p(x) u' (x))' + q(x) u(x)=z u(x)
 \label{eq:Jy}\end{equation}
 associated with operator \eqref{eq:ZP+}. To avoid inessential technical complications, we
 always suppose that $ p \in C^1 ({\mathbb R}_{+})$, $ q \in C ({\mathbb R}_{+})$ and the functions $p(x)$, $q(x)$ have finite limits as $x\to 0$.
 More general conditions on the regularity of $p(x)$ and $q(x)$ are stated, for example, in \cite[Section~15]{Nai}. We assume that $p(x)> 0$ for $x\geq 0$. The solutions of equation~\eqref{eq:Jy} exist, belong to $C^2 ({\mathbb R}_{+})$ and they have limits $u(+0)= : u(0)$, $u'(+0)= : u'(0)$. A solution $u(x)$ is distinguished uniquely by boundary conditions $u(0)=u_{0}$, $u'(0)=u_{1}$.

Recall that for arbitrary solutions $u$ and $v$ of equation \eqref{eq:Jy} their Wronskian
 \[
\{ u, v \} : = p(x) (u' (x) v (x)- u (x) v' (x))
\]
does not depend on $x\in{\mathbb R}_{+}$. Clearly, the Wronskian $\{ u, v \} =0$ if and only if the solutions $u$ and $v$ are proportional.

We introduce a couple of standard solutions of equation \eqref{eq:Jy} by boundary conditions
 \begin{equation}
 \begin{cases}
\varphi_{z}(0)= 1,&\quad \varphi'_{z}(0)=\alpha,
\\
\theta_{z}(0)= 0,& \quad \theta'_{z}(0)=- p(0)^{-1},
 \end{cases} \qquad \mbox{if} \quad \alpha\in{\mathbb R}
 \label{eq:Pz}\end{equation}
 and
 \begin{equation}
 \begin{cases}
\varphi_{z}(0)= 0,&\quad \varphi'_{z}(0)=1,
\\
\theta_{z}(0)= p(0)^{-1},&\quad \theta'_{z}(0)= 0,
 \end{cases} \qquad \mbox{if}\quad \alpha=\infty.
\label{eq:Pzi}\end{equation}
 Clearly, $\varphi_{z} (x)$ (but not $\theta_{z} (x)$) satisfies boundary condition \eqref{eq:BCz}. Note also that the Wronskian
 $\{\varphi_{z}, \theta_{z}\}=1$.

The Weyl limit point/circle theory (see, e.g., \cite[Chapter~IX]{CoLe}) states that dif\-fe\-ren\-tial equation~\eqref{eq:Jy} always has a non-trivial solution in $L^2 ({\mathbb R}_{+})$ for $\operatorname{Im} z \neq 0$. This solution is either unique (up to a constant factor) or all solutions of~\eqref{eq:Jy} belong to $L^2 ({\mathbb R}_{+})$.
The first instance is known as the limit point (LP) case and the second one~-- as the limit circle (LC) case.
In the LC case we have
 \begin{equation}
\varphi_{z}\in L^2 ({\mathbb R}_{+}),\qquad \theta_{z}\in L^2 ({\mathbb R}_{+})\qquad \mbox{for all}\quad z\in {\mathbb C}.
 \label{eq:PQz}\end{equation}

\subsection{Minimal and maximal operators}\label{section2.2}

 We first define a minimal
 operator $A_{00}$ by the equality $A_{00} u= {\mathcal A} u$ on domain
 $ {\mathcal D} (A_{00}) $ that consists of functions $u\in C ^2 ({\mathbb R}_{+}) $ such that $u(x)=0$ for sufficiently large $x$, limits $u(+0)=: u(0)$, $u'(+0)=: u'(0)$ exist and condition \eqref{eq:BCz} is satisfied.
 Thus, the boundary condition \eqref{eq:BCz} at $x=0$ is included in the definition of the operator $A_{00}$ so that its self-adjoint extensions are determined by conditions for $x\to\infty$.

The closure of $A_{00}$ will be denoted $A_{\min} $. This operator is
symmetric in the space $L^2 ({\mathbb R}_{+})$, but without additional assumptions on the coefficients $p(x)$ and $q(x)$ its domain $ {\mathcal D} (A_{\min}) $ does not admit an efficient description. The adjoint operator $A^*_{\min} =: A_{\max}$ is again given by the formula $A_{\max} u={\mathcal A} u$ on a set $ {\mathcal D} (A_{\max})$ that consists of functions
$u(x)$ belonging locally to the Sobolev space ${\sf H}^2$, satisfying boundary condition \eqref{eq:BCz} and such that $u\in L^2 ({\mathbb R}_{+})$, ${\mathcal A} u\in L^2 ({\mathbb R}_{+})$.
In the LC case, the operator $A_{\max}$ is not symmetric. Integrating by parts, we see that
 for all $u, v \in {\mathcal D} (A_{\max})$
 \[
\langle {\mathcal A} u,v \rangle - \langle u, {\mathcal A} v \rangle = \lim_{x\to\infty} p(x) \big(u'(x)\bar{v} (x)- \bar{u}' (x) v(x)\big),
\]
where the limit in the right-hand side exists but is not necessarily zero.

 Recall that
\[
 A_{\min}= A_{\min}^{**} = A_{\max}^{*}.
 \]
 The operator $A_{\min}$ is self-adjoint if and only if the LP case occurs.
 In this paper we are interested in the LC case when
 \begin{equation}
 A_{\min}\neq A_{\max}=A_{\min}^*.
 \label{eq:cl}\end{equation}

Since the operator $A_{\min} $ commutes with the complex conjugation, its deficiency indices
 \[
 d_{\pm}: =\dim\ker (A_{\max}-z I),\qquad \pm \operatorname{Im} z>0,
 \]
 are equal, i.e., $d_{+} =d_{-}=:d $, and, so, $A_{\min} $ admits self-adjoint extensions. For an arbitrary $z\in{\mathbb C}$, all solutions of equation \eqref{eq:Jy} with boundary condition~\eqref{eq:BCz}
 are given by the formula $u (x)= c \varphi_{z}(x)$ for some $c\in{\mathbb C}$. They belong to $ {\mathcal D} (A_{\max})$ if and only if $\varphi_{z} \in L^2 ({\mathbb R}_{+})$. Therefore $d=0$ if $\varphi_{z} \not\in L^2 ({\mathbb R}_{+})$ for $\operatorname{Im} z\neq 0$; otherwise $d=1$.

\subsection{Boundary conditions at infinity}\label{section2.3}

 In this paper we are interested in the case, where equation \eqref{eq:Jy} is in the limit circle (LC) case at infinity. This means that all solutions of this equation for some (and then for all) $z\in{\mathbb C}$ are in $L^2({\mathbb R}_{+})$ or, equivalently, that relation~\eqref{eq:cl} is satisfied.

First, we briefly recall the traditional description of self-adjoint extensions of the minimal operator $J_{\min}$ in terms of boundary conditions at infinity.
 We refer to the classical books \cite[Chapter~IX, Section~4]{CoLe} and \cite[Section~18]{Nai} for detailed presentations. We mention also the relatively recent book~\cite{Schm}, where a concise exposition of the case of second-order differential operators is given. The results stated below can be found, for example, in \cite[Proposition~15.14]{Schm}.

 Let $v_{j}(x)$, $j=1,2$, be some real valued functions of $x\in{\mathbb R}_{+}$ such that
 \begin{equation}
 \lim_{x\to\infty} p(x) \big(v_{1}'(x) v_2 (x)- v_{1}(x) v_{2}'(x)\big)=1.
 \label{eq:bc1}\end{equation}
 Let a set ${\mathcal D}^{(s)} \subset {\mathcal D} ( A_{\max})$ consist of functions $u(x)$ satisfying the condition
 \begin{equation}
 \lim_{x\to\infty} p(x) \big(u'(x) (s v_{1}(x) + v_{2}(x))- u(x) (s v_{1}'(x) + v_{2}'(x))\big)=0
 \label{eq:bc}\end{equation}
 if $s\in{\mathbb R}$; if $s=\infty $, then the function $s v_{1}(x) + v_{2}(x)$ in this formula should be replaced by~$v_{1} (x)$.
 Then the restriction ${\sf A}^{(s)}$ of the operator $ A_{\max}$ on domain $ {\mathcal D} ({\sf A}^{(s)}):= {\mathcal D}^{(s)} $ is self-adjoint, and each self-adjoint extension of the operator $ A_{\min}$ coincides with an operator
 ${\sf A}^{(s)}$ for some $s\in{\mathbb R}\cup\{\infty\}$.

 The resolvents of the operators ${\sf A}^{(s)}$ are determined by a formula similar to the regular case (see formula~\eqref{eq:R-LP} below). It turns out that equation~\eqref{eq:Jy}, where $\operatorname{Im} z\neq 0$ has a solution $u(x)=: f^{(s)}_{z} (x)$ satisfying boundary condition~\eqref{eq:bc}. Then, for all $ h\in L^2 ({\mathbb R}_{+})$ and $ \operatorname{Im} z\neq 0$, one has
 \[
\big( \big({\sf A}^{(s)}-z I \big)^{-1} h\big) (x) = \frac{1}{\big\{\varphi_{z}, f^{(s)}_{z}\big\}} \bigg(f_{z}^{(s)}(x) \int_{0}^x \varphi_{z}(y) h (y) \,{\rm d} y+
 \varphi_{z}(x) \int_x^\infty f_{z}^{(s)}(y) h (y) \,{\rm d}y\bigg),
 \]
where $\big\{\varphi_{z}, f^{(s)}_{z}\big\}$ is the Wronskian of the solutions $\varphi_{z}$ and $ f^{(s)}_{z}$ of equation~\eqref{eq:Jy}.

 Note that the results stated above are obtained by approximating the problem on the half-axis~${\mathbb R}_{+}$ by regular problems on intervals $(0,\ell)$ and studying the limit $\ell\to\infty$.

\subsection{Self-adjoint extensions}\label{section2.4}

 The description of self-adjoint extensions of the operator $A_{\min}$ given in the previous subsection seems to be not very efficient. In particular, it depends on a choice of the functions $v_{j}(x)$, $j=1,2$, satisfying condition \eqref{eq:bc1}. We suggest an alternative approach motivated by an analogy with Jacobi operators in Theorem~\ref{Neum2} (cf.\ \cite[Lemma~6.22 and Theorem~6.23]{Schm} or \cite[Theorem~2.6]{Simon}). In the long run, it relies on von Neumann formulas but is adapted to operators~\eqref{eq:ZP+} with real coefficients~$p(x)$ and~$q(x)$.

Our descriptions of various domains are given in terms of the solutions $\varphi_{z}(x)$ and $\theta_{z}(x)$ of differential equation \eqref{eq:Jy}. Note that the function $\varphi_{z}(x)$ satisfies boundary condition \eqref{eq:BCz} so that $\varphi_{z} \in {\mathcal D}(A_{\max})$, but this is not the case for $\theta_{z}(x)$. To get rid of this nuisance, we introduce a function $\tilde{\theta}_z (x)= \omega(x) \theta_z (x)$, where the cut-off $\omega\in C^\infty ({\mathbb R}_{+})$, $\omega(x)=0$ for small $x$ and $\omega(x)=1$ for large $x$; then $\tilde\theta_{z} \in {\mathcal D}(A_{\max})$. A direct calculation shows that
 \begin{equation}
 \big( {\mathcal A}\tilde{\theta}_z \big)(x)-z\tilde{\theta}_z (x)=\psi_{z}(x),
 \label{eq:the}\end{equation}
 where
 \begin{equation}
 \psi_{z}(x)= -p(x) \omega'(x) \theta'_{z} (x) - \big(p(x) \omega'(x) \theta_{z} (x) \big)'
 \label{eq:the1}\end{equation}
 has a compact support.

Now we are in a position to describe ${\mathcal D}(A_{\max})$.
 For a vector $h\in L^2 ({\mathbb R}_{+})$, we denote by $\{h\}$ the one dimensional subspace of $L^2 ({\mathbb R}_{+})$ spanned by the vector $h$. The symbol $\dotplus$ denotes the direct sum of subspaces.

 \begin{Theorem}\label{Neum}
 Let inclusions \eqref{eq:PQz} hold true.
Then
\begin{equation}
{\mathcal D}(A_{\max})= {\mathcal D}(A_{\min}) \dotplus \{\varphi_{0 } \}\dotplus \big\{\tilde{\theta}_{0} \big\}.
 \label{eq:Neum}\end{equation}
 \end{Theorem}

 \begin{Remark}\label{Neur}
 Since the difference of functions $\tilde{\theta}_0 $ corresponding to two different cut-offs $\omega(x)$ is in ${\mathcal D}(A_{\min})$, the direct sum in~\eqref{eq:Neum} does not depend on a particular choice of $\tilde{\theta}_0$.
 \end{Remark}

 \begin{proof}
 We start a proof of Theorem~\ref{Neum} with a direct calculation.

 \begin{Lemma}\label{Neum1}
 Suppose that
 \begin{equation}
u=u_{0}+ \alpha_{1} \varphi_{0}+ \alpha_{2} \tilde{\theta}_{0 }\qquad\mbox{and}\qquad v=v_{0}+ \beta_{1} \varphi_{0}+ \beta_{2} \tilde{\theta}_{0 },
 \label{eq:Ne1}
 \end{equation}
 where $u_{0}, v_{0} \in {\mathcal D}(A_{\min})$ and $\alpha_{j}, \beta_{j}\in {\mathbb C}$. Then
 \begin{equation}
\langle A_{\max} u,v\rangle -\langle u, A_{\max} v\rangle = \alpha_{2} \overline{\beta_{1}}-\alpha_{1} \overline{\beta_{2}}.
 \label{eq:Ne6}\end{equation}
 \end{Lemma}

 \begin{proof}
 Let us calculate
 \[
\langle A_{\max} u,v\rangle = \big\langle A_{\max} \big(u_{0}+ \alpha_{1} \varphi_{0}+ \alpha_{2} \tilde{\theta}_{0 }\big),v_{0}+ \beta_{1} \varphi_0+ \beta_{2} \tilde{\theta}_{0 }\big\rangle.
 \]
 Using \eqref{eq:the}, we see that
 \[
 A_{\max} \big(u_{0}+ \alpha_{1} \varphi_{0}+ \alpha_{2} \tilde{\theta}_{0 }\big)= A_{\min} u_{0} + \alpha_{2}\psi_{0}
 \]
 and
 \begin{align*}
\big \langle A_{\min} u_{0} ,v_{0}+ \beta_{1} \varphi_0+ \beta_{2} \tilde{\theta}_{0 }\big\rangle
 & =
 \langle A_{\min} u_{0} ,v_{0}\rangle + \big\langle u_{0} , A_{\max} \big(\beta_{1} \varphi_0+ \beta_{2} \tilde{\theta}_{0 }\big)\big\rangle
\nonumber \\
 & = \langle A_{\min} u_{0} ,v_{0}\rangle + \overline{\beta_{2}} \langle u_{0} , \psi_{0}\rangle,
 \end{align*}
 whence
 \[
\langle A_{\max} u,v\rangle = \langle A_{\min} u_{0} ,v_{0}\rangle + \overline{\beta_{2}} \langle u_{0} , \psi_{0}\rangle
+ \alpha_{2} \langle \psi_{0},v_{0}\rangle
+ \alpha_{2}\overline{\beta_{1}} \langle \psi_{0} , \varphi_{0}\rangle
+ \alpha_{2}\overline{\beta_{2}} \big\langle \psi_{0} , \tilde{\theta}_{0}\big\rangle .
 \]
 Similarly, we find that
 \[
\langle u, A_{\max} v\rangle = \langle u_{0} , A_{\min}v_{0}\rangle + \alpha_{2} \langle \psi_{0}, v_{0}\rangle
+ \overline{\beta_{2}}\langle u_{0}, \psi_{0}\rangle
+ \overline{\beta_{2}}\alpha_{1}\langle \varphi_{0}, \psi_{0}\rangle
+ \overline{\beta_{2}}\alpha_{2} \big\langle \tilde{\theta}_{0}, \psi_{0}\big\rangle .
 \]
 Comparing the last two equalities and taking into account that the functions $\psi_{0}$, $\varphi_{0}$ and $\tilde{\theta}_{0}$ are real, we see that
 \begin{equation}
\langle A_{\max} u,v\rangle - \langle u, A_{\max} v \rangle= \big(\alpha_{2} \overline{\beta_{1}} - \alpha_{1} \overline{\beta_{2}} \big)\langle \varphi_{0}, \psi_{0}\rangle.
 \label{eq:Ne5}\end{equation}
 Using definition \eqref{eq:the1} and integrating by parts, it is easy to calculate
 \[
 \langle \varphi_{0}, \psi_{0}\rangle=\int_{0}^\infty p(x) \omega'(x) \big(\theta_{0} (x)\varphi_{0}' (x) - \theta'_{0} (x)\varphi_{0} (x)
 \big) \,{\rm d}x.
 \]
 Since $\{ \varphi_{0}, \theta_{0}\}=1$, this integral equals $1$. Therefore identity~\eqref{eq:Ne5} can be rewritten as~\eqref{eq:Ne6}.
 \end{proof}

 Now it is easy to prove Theorem~\ref{Neum}. First we check that the sum in the right-hand side of~\eqref{eq:Neum} is direct, that is, an inclusion
 $
\alpha_{1} \varphi_{0}+ \alpha_{2} \tilde{\theta}_{0 }\in {\mathcal D}(A_{\min})
$
 implies that $\alpha_{1}= \alpha_{2} =0$. Indeed, if this inclusion is true, then
 \[
 \big\langle A_{\max} \big(\alpha_{1} \varphi_{0}+ \alpha_{2} \tilde{\theta}_{0 }\big), \beta_{1} \varphi_{0}+ \beta_{2} \tilde{\theta}_{0 }\big\rangle = \big\langle \alpha_{1} \varphi_{0}+ \alpha_{2} \tilde{\theta}_{0 }, A_{\max} \big(\beta_{1} \varphi_{0}+ \beta_{2} \tilde{\theta}_{0 }\big)\big\rangle
 \]
 for all $\beta_{1}, \beta_{2}\in {\mathbb C}$. Therefore it follows from Lemma~\ref{Neum1} for the particular case $u_{0}=v_{0}=0$ that $ \alpha_{2} \overline{\beta_{1}}-\alpha_{1} \overline{\beta_{2}}=0$ whence $ \alpha_{1} = \alpha_{2}=0$ because $ \beta_{1} $ and $ \beta_{2}$ are arbitrary.

 Obviously, the right-hand side of~\eqref{eq:Neum} is contained in its left-hand side. Actually, there is the equality here because the operator $A_{\min}$ has deficiency indices $(1,1)$ so that the dimension of the factor space $ {\mathcal D}(A_{\max}) / {\mathcal D}(A_{\min}) $ equals~$2$. This concludes the proof of Theorem~\ref{Neum}.
 \end{proof}

All self-adjoint extensions $A_{t}$ of the operator $A_{\min}$ are parametrized by numbers $ t\in {\mathbb R}$ and $t=\infty$. Let sets ${\mathcal D}(A_{t})\subset {\mathcal D}(A_{\max})$ be distinguished by conditions
 \begin{equation}
{\mathcal D}(A_{t})= {\mathcal D}(A_{\min})\dotplus \big\{ t \varphi_{0} + \tilde{\theta}_{0}\big\}, \qquad t\in {\mathbb R},
 \label{eq:Neum1}\end{equation}
and
 \begin{equation}
{\mathcal D}(A_\infty)= {\mathcal D}(A_{\min})\dotplus \{ \varphi_{0 } \}.
 \label{eq:Neum2}\end{equation}

\begin{Theorem}\label{Neum2}
 Let inclusions \eqref{eq:PQz} hold true. Then all operators $A_{t}$ are self-adjoint. Conversely, every operator $A$ satisfying condition
 \eqref{eq:Neum3}
coincides with one of the operators $A_{t}$ for some $t\in{\mathbb R}\cup \{\infty\}$.
 \end{Theorem}

 \begin{proof}
 We proceed from Lemma~\ref{Neum1}. Let $u,v \in{\mathcal D}(A_{\max})$ so that equalities \eqref{eq:Ne1} are satisfed.
 If $u,v \in{\mathcal D}(A_{t})$, then according to \eqref{eq:Neum1} or \eqref{eq:Neum2} we have $\alpha_{1}= t\alpha_{2}$, $\beta_{1}= t \beta_{2}$ if $t\in {\mathbb R}$ and $\alpha_2= \beta_{2}=0$ if $t=\infty$. Therefore it follows from relation \eqref{eq:Ne6} that $\langle A_{t}u,v \rangle= \langle u, A_{t} v \rangle$, and hence the operators $A_{t} $ are symmetric.

 If $v\in {\mathcal D}(A_{t}^*)$, then $\langle A_{t}u,v \rangle= \langle u, A_{t} v \rangle$ for all $u\in {\mathcal D}(A_{t})$.
 Thus, according again to \eqref{eq:Ne6}, $\alpha_{2} \overline{\beta_{1}}-\alpha_{1} \overline{\beta_{2}}=0$ for all $\alpha_{1}$, $\alpha_{2}$ such that $\alpha_{1}=t \alpha_{2}$ if $t\in {\mathbb R}$ and such that $ \alpha_{2}= 0$ if $t=\infty$.
 Let first $t\in {\mathbb R}$. Then $\alpha_{2} \big( \overline{\beta_{1}}- t \overline{\beta_{2}}\big)=0 $ whence $\beta_{1}=t\beta_{2}$ because $\alpha_{2}$ is arbitrary. If $t=\infty$, we have $\alpha_{1} \overline{\beta_{2}}=0$ whence $\beta_{2}=0$ because $\alpha_1$ is arbitrary. It follows that $v\in {\mathcal D}(A_{t})$, and consequently $A_{t}=A_{t}^*$.

Suppose that an operator $A$ satisfies \eqref{eq:Neum3}. Since $A$ is symmetric, it follows from Lemma~\ref{Neum1} that $\alpha_{2} \overline{\beta_{1}}=\alpha_{1} \overline{\beta_{2}}$ for all $u, v\in {\mathcal D} (A)$ and the corresponding coefficients $\alpha_{j}$, $\beta_{j}$ defined in \eqref{eq:Ne1}. Suppose that $\alpha_{2}\neq 0$ for some $u \in {\mathcal D} (A)$. Then setting $u=v$, we see that $\alpha_{2} \overline{\alpha_{1}}=\alpha_{1} \overline{\alpha_{2}}$ whence $ \alpha_{1}\alpha_{2}^{-1}=:t \in{\mathbb R}$. Now equality $\alpha_{2} \overline{\beta_{1}}=\alpha_{1} \overline{\beta_{2}}$ implies that
 $\beta_{1}= t \beta_{2}$ for all $v \in {\mathcal D} (A)$ so that $A=A_{t}$. If $\alpha_{2}= 0$ for all $u \in {\mathcal D} (A)$, then $A=A_\infty$.
 \end{proof}

\section{Resolvents of self-adjoint extensions}\label{section3}

 Our goal in this section is to construct resolvents of the operators $A_{t}$. We start however with a~construction of a similar object for the operator $A_{\max}$.

\subsection{Quasiresolvent of the maximal operator}\label{section3.1}

Recall that in the LC case
 inclusions~\eqref{eq:PQz} are satisfied. Let us define, for all $z\in {\mathbb C}$, a~bounded operator ${\mathcal R} (z)$ in the space $L^2 ({\mathbb R}_{+})$ by the equality
 \begin{equation}
 ( {\mathcal R} (z)h) (x) = \theta_{z} (x) \int_{0}^x \varphi_{z} (y) h(y) \,{\rm d}y+
\varphi_{z} (x) \int_x^\infty \theta_{z} (y) h(y) \,{\rm d} y .
 \label{eq:RR11}\end{equation}
We prove (see Theorem~\ref{res}) that, in a natural sense, ${\mathcal R} (z)$ can be considered as a quasiresolvent of the operator~$A_{\max}$. It plays the role of the resolvent of the operator~$A_{\max}$.

Let us enumerate some simple properties of the operator ${\mathcal R} (z)$.
 Obviously, the operator ${\mathcal R} (z)$ belongs to the Hilbert--Schmidt class. It depends analytically on $z\in {\mathbb C}$ and ${\mathcal R} (z)^*={\mathcal R} (\bar{z})$. Differentiating
 definition \eqref{eq:RR11}, we see that
 \begin{equation}
 ( {\mathcal R} (z)h)' (x) = \theta_{z}' (x) \int_{0}^x \varphi_{z} (y) h(y) \,{\rm d}y+
 \varphi_{z}' (x) \int_x^\infty \theta_{z} (y) h(y) \,{\rm d}y
 \label{eq:RR12}\end{equation}
 for all $h\in L^2 ({\mathbb R}_{+})$.
 In particular, it follows from relations~\eqref{eq:RR11} and~\eqref{eq:RR12} that
 \begin{equation}
({\mathcal R} (z)h)(0)= \varphi_{z} (0)\langle h, \theta_{\bar{z}}\rangle
\label{eq:r01}\end{equation}
and
 \begin{equation}
({\mathcal R} (z)h)'(0)= \varphi_{z}' (0)\langle h, \theta_{\bar{z}}\rangle,
\label{eq:r02}\end{equation}
where $\varphi_{z} (0)$ and $\varphi_{z}' (0)$ are defined by equalities \eqref{eq:Pz} or \eqref{eq:Pzi}.

A proof of the following statement is close to the construction of the resolvent for essentially self-adjoint Schr\"odinger operators.

 \begin{Theorem}\label{res}
 Let inclusions \eqref{eq:PQz} hold true.
 For all $z \in {\mathbb C}$, we have
 \begin{equation}
{\mathcal R} (z)\colon \ L^2 ({\mathbb R}_{+})\to {\mathcal D} (A_{\max} )
\label{eq:qres}\end{equation}
and
 \begin{equation}
(A_{\max} -zI) {\mathcal R} (z)=I.
\label{eq:qres1}\end{equation}
 \end{Theorem}

 \begin{proof}
 Let $h\in L^2 ({\mathbb R}_{+})$ and $u(x)= ({\mathcal R} (z)h)(x)$. Boundary condition \eqref{eq:BCz} is a direct consequence of relations \eqref{eq:r01} and \eqref{eq:r02}. Differentiating \eqref{eq:RR12}, we see that
 \begin{gather}
 (p(x)u' (x))' = (p(x) \theta_{z}' (x) )'\int_{0}^x \varphi_{z} (y) h(y) \,{\rm d}y+ (p(x) \varphi_{z}' (x) )' \int_x^\infty \theta_{z} (y) h(y) \,{\rm d}y \nonumber \\
 \hphantom{(p(x)u' (x))' =}{} + p(x) \big( \theta'_{z}(x) \varphi_{z}(x)-\theta
 _{z}(x) \varphi'_{z}(x)\big) h(x).
 \label{eq:RR13}\end{gather}
Since the Wronskian $\{ \varphi_{z}, \theta_{z}\}=1 $, the last term in the right-hand side equals $-h(x)$. Putting now equalities~\eqref{eq:RR11} and~\eqref{eq:RR13} together and using equation~\eqref{eq:Jy} for the functions $\varphi_{z}(x)$ and $\theta_{z}(x)$, we obtain the equation
\[
- (p(x) u' (x))' + q(x) u(x)-z u(x)=h(x) .
 \]
 Taking also into account boundary condition \eqref{eq:BCz}, we see that $ A_{\max} u -z u=h $. Since $h\in L^2 ({\mathbb R}_{+})$, this yields both \eqref{eq:qres} and \eqref{eq:qres1}.
 \end{proof}

 \begin{Remark}\label{res3}
 In definition \eqref{eq:RR11}, only boundary condition \eqref{eq:BCz} for $\varphi_{z} (x)$ and the relation $\{ \varphi_{z}, \theta_{z}\}=1 $ for the Wronskian are essential. For example, one can replace the solution $\theta_{z}(x)$ by $\theta_{z}(x)+ \delta \varphi_{z} (x)$ for some $\delta \in {\mathbb C}$. Then the operator
 $ {\mathcal R} (z)$ will be replaced by
 $\widetilde{\mathcal R} (z)= {\mathcal R} (z) +\delta \langle \cdot, \varphi_{\bar{z} }\rangle \varphi_{z}$ and formulas \eqref{eq:qres}, \eqref{eq:qres1} remain true for $\widetilde{\mathcal R} (z)$.
 \end{Remark}

 Note that solutions $u(x)$ of differential equation \eqref{eq:Jy} satisfying condition \eqref{eq:BCz} are given by the formula $u(x)= \Gamma \varphi_{z} (x)$ for some $ \Gamma \in{\mathbb C}$. Therefore we
 can state

 \begin{Corollary}\label{res1}
 All solutions of the equation
 \[
(A_{\max} -zI) u =h, \qquad \mbox{where} \quad z\in {\mathbb C} \quad \mbox{and}\quad h\in L^2 ({\mathbb R}_{+}),
\]
 for $u\in {\mathcal D} (A_{\max} )$ are given by the formula
 \begin{equation}
 u = \Gamma \varphi_{z}+ {\mathcal R} (z) h \qquad \mbox{for some} \quad \Gamma=\Gamma (z;h) \in{\mathbb C}.
\label{eq:qres3}\end{equation}
 \end{Corollary}

A relation below is a direct consequence of definition \eqref{eq:RR11}
 and condition \eqref{eq:PQz}:
 \[
 ( {\mathcal R} (z)h) (x) = \theta_{z} (x) \langle h, \varphi_{\bar{z}} \rangle+ o(|\varphi_{z} (x)| +|\theta_{z} (x)| )\qquad {\rm as}\quad x\to\infty.
 \]

 This asymptotic formula can be supplemented by the following result.

 \begin{Proposition}\label{AS}
 For all $z\in {\mathbb C}$ and all $h\in L^2 ({\mathbb R}_{+})$, we have
 \begin{equation}
 u:= {\mathcal R} (z)h - \tilde{\theta}_{z} \langle h, \varphi_{ \bar{z}} \rangle\in {\mathcal D}(A_{\min}).
 \label{eq:Ras1}\end{equation}
 \end{Proposition}

\begin{proof}
If the support of $h(x)$ is compact in ${\mathbb R}_{+}$, then $ ( {\mathcal R} (z)h) (x) = \varphi_{z}(x) \langle h, \theta_{\bar{z}} \rangle$ for sufficiently small $x$ and $ ( {\mathcal R} (z)h) (x) = \theta_{z}(x) \langle h, \varphi_{\bar{z}} \rangle$ for sufficiently large $x$. Therefore $u(x)$ satisfies boundary condition \eqref{eq:BCz} at $x=0$ and $u(x)=0$ for large $x$ whence $u\in{\mathcal D}(A_{00})$.

Let now $h$ be an arbitrary vector in $L^2 ({\mathbb R}_{+})$. Observe that $ u \in {\mathcal D}(A_{\min})$ if and only if there exists a sequence $u^{(k)}\in {\mathcal D}(A_{00})$ such that
 \begin{equation}
u^{(k)}\to u \qquad \mbox{and} \qquad {\mathcal A}u^{(k)}\to {\mathcal A}u
 \label{eq:Ras2}\end{equation}
 in $L^2 ({\mathbb R}_{+})$ as $ k\to\infty$.
 Let us take a sequence of functions $h^{(k)} $ with compact supports in ${\mathbb R}_{+}$
 such that $h^{(k)}\to h$ and set
 \[
 u^{(k)}= {\mathcal R} (z)h^{(k)} - \big \langle h^{(k)}, \varphi_{\bar{z}} \big\rangle \tilde{\theta}_{z} .
 \]
 Then, as was already shown, $u^{(k)} \in {\mathcal D} (A_{00})$ and $u^{(k)}\to u$ as $k\to\infty$ because the operator $ {\mathcal R} (z)$ is bounded. It follows from formula \eqref{eq:qres1} that
 \[
({\mathcal A} -z) u^{(k)}= h^{(k)} - \big\langle h^{(k)}, \varphi_{\bar{z}} \big\rangle ({\mathcal A} -z) \tilde{\theta}_{z} \to h - \langle h, \varphi_{\bar{z}} \rangle ({\mathcal A} -z) \tilde{\theta}_{z}
\]
as $k\to\infty$. The right-hand side equals $({\mathcal A} -z) u$ by formula~\eqref{eq:qres1} and definition~\eqref{eq:Ras1}. This proves relations~\eqref{eq:Ras2} whence $u\in {\mathcal D}(A_{\min})$.
 \end{proof}

\subsection{Resolvent representation}\label{section3.2}

First, we find a link between the solutions $\varphi_{z}$, $\theta_{z}$ of equation \eqref{eq:Jy} for an arbitrary $z \in {\mathbb C}$ and for $z=0$.

\begin{Lemma}\label{PQ}
 For all $z \in {\mathbb C}$, we have
 \begin{equation}
\varphi_{z} - z {\mathcal R}(0) \varphi_{z}=\big( 1-z \langle \varphi_{z}, \theta_{0}\rangle \big) \varphi_{0}
\label{eq:PP1}\end{equation}
and
 \begin{equation}
\theta_{z} - z {\mathcal R}(0) \theta_{z}= - z \langle \theta_{z}, \theta_{0} \rangle \varphi_{0} + \theta_{0}.
\label{eq:QQ}\end{equation}
 \end{Lemma}

\begin{proof}
 To prove \eqref{eq:PP1}, we set
 $ u= \varphi_{z} -z {\mathcal R}(0) \varphi_{z} $
and observe that $ {\mathcal A } u= {\mathcal A }\varphi_{z}-z \varphi_{z} =0$ according to equation \eqref{eq:Jy} for $ \varphi_{z} $ and relation \eqref{eq:qres1} (where $z=0$). Since both $ \varphi_{z} (x) $ and $( {\mathcal R}(0) \varphi_{z} ) (x) $ satisfy boundary condition \eqref{eq:BCz}, it follows that $u(x)=c \varphi_{0} (x)$ and hence
 \begin{equation}
\varphi_{z}(x) - z ({\mathcal R}(0) \varphi_{z} )(x)= c \varphi_{0} (x)
\label{eq:PP1A}\end{equation}
for some constant $c\in{\mathbb C}$.
It remains to find this constant. If $\alpha\in {\mathbb R}$, we set $x=0$. Then $\varphi_{z} (0)=\varphi_0 (0)=1$ and
$({\mathcal R}(0) \varphi_{z} )(0)$ is given by \eqref{eq:r01} so that \eqref{eq:PP1A} for $x=0$ yields $c= 1-z \langle \varphi_{z}, \theta_{0}\rangle$. This proves \eqref{eq:PP1}. In the case
 $\alpha=\infty$, we first differentiate \eqref{eq:PP1A} and then set $x=0$. Since $\varphi_{z}' (0)=\varphi_0' (0)=1$,
using \eqref{eq:r02} we again get equality \eqref{eq:PP1}.

The proof of \eqref{eq:QQ} is quite similar. We now set
 \begin{equation}
 v=\theta_{z} - \theta_{0} -z {\mathcal R}(0) \theta_{z}
\label{eq:QQ1}\end{equation}
and find that $ {\mathcal A } v= 0$ according to equation \eqref{eq:Jy} for $\theta_{z} $ and relation \eqref{eq:qres1} (where $z=0$). Next, we observe that
 $v (0)= -z \varphi_{z}(0)\langle \theta_{z}, \theta_{0}\rangle$ because $\theta_{z} (0)= \theta_{0} (0)$ and $ ({\mathcal R}(0) \theta_{z})(0)$ is given by~\eqref{eq:r01}. Similarly, it follows from equalities $\theta_{z}' (0)= \theta_{0}' (0)$ and \eqref{eq:r02} that $v' (0)= -z \varphi_{z}'(0)\langle \theta_{z}, \theta_{0}\rangle$. Thus, $v(x)$ satisfies equation~\eqref{eq:Jy}
 and boundary condition~\eqref{eq:BCz} whence $v(x)=c \varphi_{0} (x)$ for some constant $c\in{\mathbb C}$.
 In the case
 $\alpha\in {\mathbb R}$ we use this equality for $x=0$ and in the case
 $\alpha=\infty$ we use that $v'(x)=c \varphi_{0}' (x)$. In both cases we obtain that $c=- z \langle \theta_{z}, \theta_{0} \rangle$. In view of \eqref{eq:QQ1} this ensures~\eqref{eq:QQ}.
\end{proof}

 Putting together Lemma~\ref{PQ} with Proposition~\ref{AS} (for $z=0$), we can also state the following result.

 \begin{Lemma}\label{PQ+}
 For all $z \in {\mathbb C}$, we have
 \[
\varphi_{z} -\big( 1-z \langle \varphi_{z}, \theta_{0} \rangle \big) \varphi_{0} - z \langle \varphi_{z}, \varphi_0\rangle \tilde\theta_{0} \in {\mathcal D} (A_{\min})
\]
and
 \[
\tilde\theta_{z} + z \langle \theta_{z}, \theta_0\rangle \varphi_0 - (1+z \langle \theta_{z}, \varphi_{0}\rangle ) \tilde\theta_{0} \in {\mathcal D} (A_{\min}).
\]
 \end{Lemma}

 Now we are in a position to construct the resolvents of the self-adjoint operators $A_t$.

 \begin{Theorem}\label{RES}
 Let inclusions \eqref{eq:PQz} hold true.
 For all $z\in {\mathbb C}$ with $\operatorname{Im} z\neq 0$ and all $h\in L^2 ({\mathbb R}_{+})$, the resolvent $R_t (z)= (A_t-zI)^{-1}$ of the operator $A_t$ is given by an equality
 \begin{equation}
R_{t} (z) h = \gamma_{t} (z) \langle h, \varphi_{\bar{z}}\rangle \varphi_{z}+ {\mathcal R}(z)h,
 \label{eq:RES1}\end{equation}
 where
 \begin{equation}
\gamma_{t}(z)= \frac{z \langle \theta_{z}, \theta_{0}\rangle +\big( 1+ z \langle \theta_{z}, \varphi_{0}\rangle \big)t }
{1- z \langle \varphi_{z}, \theta_{0}\rangle -z \langle \varphi_{z}, \varphi_{0}\rangle t } \qquad \mbox{if} \quad t\in {\mathbb R}
 \label{eq:RES2}\end{equation}
 and
 \begin{equation}
\gamma_{\infty}(z)= - \frac{ 1+ z \langle \theta_{z}, \varphi_{0}\rangle }
{z \langle \varphi_{z}, \varphi_{0}\rangle }.
 \label{eq:RES3}\end{equation}
 \end{Theorem}

\begin{proof} According to Theorem~\ref{res} and Corollary~\ref{res1} a vector $u=R_{t} (z) h$ is given by equali\-ty~\eqref{eq:qres3}, where
 $ \Gamma= \Gamma_{t}(z;h)$ is a bounded linear functional of $h\in L^2 ({\mathbb R}_{+})$ so that
 $\Gamma_{t}(z;h)= \big\langle h, f_{z}^{(t)}\big\rangle$
 for some vector $ f_{z}^{(t)} \in L^2 ({\mathbb R}_{+})$.
 Since $R_{t} (z)^*=R_{t} (\bar{z})$ and ${\mathcal R} (z)^*={\mathcal R} (\bar{z})$, we see that
 \[
 \big\langle h, f_{\bar{z}}^{(t)} \big\rangle \varphi_{\bar{z}} =\langle h, \varphi_{z}\rangle f_{z}^{(t)}
 \]
 for all $h\in L^2 ({\mathbb R}_{+})$. It follows that $f_{z}^{(t)} = \overline{\gamma_{t} (z)} \varphi_{\bar{z}}$ for some $\gamma_{t} (z) \in{\mathbb C}$. This yields representation~\eqref{eq:RES1}, where the constant $\gamma_{t} (z) $ is determined by the condition
 \begin{equation}
 R_{t} (z) h\in {\mathcal D} (A_{t}).
 \label{eq:RER}\end{equation}

 Let us show that this inclusion leads to expressions~\eqref{eq:RES2} or~\eqref{eq:RES3} for $\gamma_{t}(z)$.
 By definitions~\eqref{eq:Neum1} or~\eqref{eq:Neum2} of the set ${\mathcal D} (A_{t})$, inclusion \eqref{eq:RER}
 means that
 \begin{gather}
R_{t} (z) h - X \big( t \varphi_{0} + \tilde{\theta}_{0}\big)\in {\mathcal D} (A_{\min}) \quad \mbox{if} \ t\in {\mathbb R}
 \qquad \mbox{and} \qquad R_\infty (z) h - X \varphi_{0} \in {\mathcal D} (A_{\min})
 \label{eq:REA}\end{gather}
 for some number $X=X_{t}(z)\in {\mathbb C}$. On the other hand, it follows from
 relations \eqref{eq:Ras1} and \eqref{eq:RES1} that
 \begin{equation}
R_{t} (z) h -\langle h, \varphi_{\bar{z}}\rangle \big( \gamma_{t} (z) \varphi_{z} + \tilde{\theta}_{z}\big)\in{\mathcal D} (A_{\min}).
 \label{eq:RES4}\end{equation}
 Comparing \eqref{eq:REA} and \eqref{eq:RES4}, we see that \eqref{eq:RER} is equivalent to inclusions
 \begin{equation}
\langle h, \varphi_{\bar{z}}\rangle \big(\gamma_{t} (z) \varphi_{z} + \tilde{\theta}_{z} \big) - X \big( t \varphi_{0} + \tilde{\theta}_{0}\big)\in {\mathcal D} (A_{\min})\qquad \mbox{if} \quad t\in {\mathbb R}
 \label{eq:REA1}\end{equation}
 and
 \begin{equation}
\langle h, \varphi_{\bar{z}}\rangle \big(\gamma_\infty (z) \varphi_{z}+ \tilde{\theta}_{z} \big) - X \varphi_{0} \in {\mathcal D} (A_{\min})
 \label{eq:REA2}\end{equation}

 Note that $\langle h, \varphi_{\bar{z}}\rangle\neq 0$ because the sum in \eqref{eq:Neum} is direct and set
 $Y= \langle h, \varphi_{\bar{z}}\rangle^{-1}X$.
 It follows from Lemma~\ref{PQ+} that
 inclusion \eqref{eq:REA1} is equivalent to an equality
 \begin{gather*}
 \!\gamma_{t} (z) \big( ( 1-z \langle \varphi_{z}, \theta_{0}\rangle ) \varphi_{0} +z \langle \varphi_{z}, \varphi_{0}\rangle \tilde{\theta}_{0}\big)
 + \big({-} z \langle \theta_{z}, \theta_{0} \rangle \varphi_{0} + ( 1+ z \langle \theta_{z}, \varphi_{0}\rangle ) \tilde{\theta}_{0} \big) = Y \big(t \varphi_{0} + \tilde{\theta}_{0}\big) .
 \end{gather*}
 Comparing here the coefficients at $\varphi_{0}$ and $ \tilde{\theta}_{0}$, we obtain equations
 \begin{gather*}
 \gamma_{t} (z) ( 1-z \langle \varphi_{z}, \theta_{0}\rangle )- z \langle \theta_{z}, \theta_{0}\rangle =t Y,
 \\
 \gamma_{t} (z) z \langle \varphi_{z}, \varphi_{0}\rangle + 1+ z \langle \theta_{z}, \varphi_{0}\rangle=Y,
 \end{gather*}
 which yield
 \[
 \frac{ \gamma_{t} (z) ( 1-z \langle \varphi_{z}, \theta_{0}\rangle )- z \langle \theta_{z}, \theta_{0}\rangle }
 {\gamma_{t} (z) z \langle \varphi_{z}, \varphi_{0}\rangle + 1+ z \langle \theta_{z}, \varphi_{0}\rangle } =t.
 \]
 Solving this equation with respect to $\gamma_{t} (z) $, we arrive at formula \eqref{eq:RES2}.
 Similarly, using again Lemma~\ref{PQ+}, we see that
 inclusion \eqref{eq:REA2} is equivalent to an equality
 \[
 \gamma_\infty (z) \big(( 1-z \langle \varphi_{z}, \theta_{0}\rangle ) \varphi_{0} + z \langle \varphi_{z}, \varphi_{0}\rangle \tilde{\theta}_{0} \big)
 + \big({-} z \langle \theta_{z}, \theta_{0} \rangle \varphi_{0} + ( 1+ z \langle \theta_{z}, \varphi_{0}\rangle ) \tilde{\theta}_{0} \big) = Y \varphi_{0} .
 \]
Inclusion \eqref{eq:REA2} holds true if and only if the coefficient at $\tilde{\theta}_{0} $ equals zero. This yields formula~\eqref{eq:RES3}.
\end{proof}

 \begin{Corollary}\label{RESb}
 If $z\in {\mathbb C}$ is a regular point of the operator $A_{t}$, then its resolvent~$R_{t} (z)$ is in the Hilbert--Schmidt class.
 In particular, the spectra of all operators~$A_{t}$ are discrete.
 \end{Corollary}

The result of this corollary is well known. It follows, for example, from Theorem~1 in Section~19.1 of the book \cite{Nai}.

We emphasize that, for different~$t$, the resolvents $R_{t}(z)$ of the operators $A_{t}$ differ from each other only by the coefficient $\gamma_{t}(z)$ at the rank one operator $\langle\cdot, \varphi_{\bar{z}}\rangle \varphi_{z}$.
 This is consistent with the fact that the operator $A_{\min}$ has deficiency indices $(1,1)$.
 Observe also that
 $ \overline{\gamma_{t}(z)}= \gamma_t (\bar{z})$.

 The functions $z \langle \theta_{z}, \theta_{0}\rangle$, $-1 +z \langle \varphi_{z}, \theta_{0}\rangle$, $ 1+ z \langle \theta_{z}, \varphi_{0}\rangle$ and $z \langle \varphi_{z}, \varphi_{0}\rangle$ in formulas~\eqref{eq:RES2} and~\eqref{eq:RES3} play the role of Nevanlinna's functions (denoted usually $A$, $B$, $C$ and~$D$) in the theory of Jacobi operators.

\subsection{Spectral measure}\label{section3.3}

In view of the spectral theorem, Theorem~\ref{RES} yields a representation for the Cauchy--Stieltjes transform of the spectral measure ${\rm d}E_{t} (\lambda)$ of the operator $A_{t}$.

 \begin{Theorem}\label{RESc}
 Let inclusions \eqref{eq:PQz} hold true.
 Then for all $z\in {\mathbb C}$ with $\operatorname{Im} z\neq 0$ and all $h\in L^2 ({\mathbb R}_{+})$, we have an equality
 \begin{equation}
 \int_{-\infty}^\infty
 (\lambda-z)^{-1} \,{\rm d} (E_{t}(\lambda)h,h)= \gamma_{t} (z)|\langle \varphi_{z},h\rangle |^2 + ({\mathcal R}(z) h,h).
 \label{eq:E1}\end{equation}
 \end{Theorem}

 Recall that the operators ${\mathcal R}(z)$ are defined by formula \eqref{eq:RR11}. Therefore $({\mathcal R}(z) h,h)$ are entire functions of $z\in{\mathbb C}$, and the singularities of the integral in \eqref{eq:E1} are determined by the function~$\gamma_{t} (z)$. Thus, \eqref{eq:E1} can be considered as a modification of the classical Nevanlinna formula (see his original paper \cite{Nevan} or, for example, formula (7.6) in the book \cite{Schm}) for the Cauchy--Stieltjes transform of the spectral measure in the theory of Jacobi operators. We mention however that, for Jacobi operators acting in the space $\ell^2 ({\mathbb Z}_{+})$, there is the canonical choice of a generating vector and of a spectral measure. This is not the case for differential operators in $L^2 ({\mathbb R}_{+})$.

 Let us discuss spectral consequences of Theorem~\ref{RES}. Since the functions $ \langle \varphi_{z}, \varphi_{0}\rangle $ and $ \langle \varphi_{z}, \theta_{0}\rangle $ are entire, it again follows from \eqref{eq:RES2} and \eqref{eq:RES3} that the spectra of the operators $A_{t}$ are discrete. Theorem~\ref{RES} yields also an equation for their eigenvalues.

 \begin{Theorem}\label{RESp}
 Let inclusions \eqref{eq:PQz} hold true.
 Then eigenvalues $\lambda $ of the operators $A_{t}$ are given by the equations
 \begin{equation}
 1 - \lambda \langle \varphi_{\lambda}, \varphi_{0}\rangle t - \lambda \langle \varphi_{\lambda}, \theta_{0}\rangle =0 \qquad \mbox{if}\quad t\in {\mathbb R}
 \label{eq:RESp2}\end{equation}
 and
 \begin{equation}
\lambda \langle \varphi_{\lambda}, \varphi_{0}\rangle =0 \qquad \mbox{if} \quad t=\infty.
 \label{eq:RES3p}\end{equation}
 \end{Theorem}

 This assertion is a modification of a R.~Nevanlinna's result obtained by him for Jacobi ope\-rators.

 We finally note an obvious fact: if $\lambda $ is an eigenvalue of an operator $A_{t}$, then the corresponding eigenfunction equals $c\varphi_{\lambda} (x)$, where $c\in{\mathbb C}$. In particular, this implies that all eigenvalues of the operators~$A_{t}$ are simple.

\subsection{Concluding remarks}\label{section3.4}

Here are some final observations.

 A. Equations \eqref{eq:RESp2} and \eqref{eq:RES3p} can be rewritten in terms of asymptotics for $x\to\infty$ of solutions to equations \eqref{eq:Jy} for $x\to\infty$.
 Multiplying differential equation \eqref{eq:Jy} for $\varphi_{\lambda}
 $ by $\varphi_{0}$, integrating over a bounded interval $(0,x)$ and then integrating by parts, we find that
 \begin{gather*}
 \lim_{x\to\infty} \bigg( {-}p(x) \varphi_{\lambda}' (x)\varphi_{0}(x)+\int_{0}^x p(y) \varphi_{\lambda}' (y)\varphi_{0}' (y)\,{\rm d}y
 +\int_{0}^x q(y) \varphi_{\lambda} (y) \varphi_{0} (y)\,{\rm d}y\bigg)
 \\
 \qquad{} + p(0) \varphi_{\lambda}' (0)\varphi_{0}(0) =\lambda\langle \varphi_{\lambda},\varphi_{0}\rangle.
 \end{gather*}
 Similarly, multiplying equation~\eqref{eq:Jy} for $\varphi_{0}$ by $\varphi_{\lambda}$, integrating over a bounded interval $(0,x)$ and then integrating by parts, we see that
 \begin{gather*}
 \lim_{x\to\infty} \bigg( {-}p(x) \varphi_{0}' (x)\varphi_{\lambda}(x)+\int_{0}^x p(y) \varphi_{0}' (y) \varphi_{\lambda}' (y)\,{\rm d}y
 +\int_{0}^x q(y) \varphi_{0} (y)\varphi_{\lambda} (y)\,{\rm d}y\bigg)
 \\
 \qquad{} + p(0) \varphi_{0}' (0)\varphi_{\lambda}(0)=0.
 \end{gather*}
 Comparing these two formulas and taking into account boundary condition \eqref{eq:BCz}, we find that
 \begin{equation}
 \lambda\langle \varphi_{\lambda},\varphi_{0}\rangle = \lim_{x\to\infty} p(x) \big( \varphi_{\lambda} (x)\varphi_{0}'(x) - \varphi_{\lambda}' (x)\varphi_{0}(x)\big).
 \label{eq:Na1}\end{equation}
 Thus, equation \eqref{eq:RES3p} is satisfied if and only if the right-hand side of \eqref{eq:Na1} is zero.

 The scalar product $\langle \varphi_{\lambda}, \theta_{0}\rangle$ can be calculated in an analogous way. We only have to observe that according to \eqref{eq:Pz} or \eqref{eq:Pzi}
 \[
 p(0) \varphi_{\lambda}'(0) \theta_{0} (0) - p(0) \theta_{0}' (0)\varphi_{\lambda}(0) =1,
 \]
 so that instead of \eqref{eq:Na1} we now have
 \begin{equation}
 \lambda\langle \varphi_{\lambda},\theta_{0}\rangle = 1+ \lim_{x\to\infty} p(x) \big( \varphi_{\lambda} (x)\theta_{0}'(x) - \varphi_{\lambda}' (x)\theta_{0}(x)\big).
 \label{eq:Na2}\end{equation}
 Putting together \eqref{eq:Na1} and \eqref{eq:Na2}, we see that equation \eqref{eq:RESp2} for $\lambda$ can be written as
 \[
 \lim_{x\to\infty} p(x) \big( \varphi_{\lambda} (x)\big(\theta_{0}'(x) + t\varphi_{0}'(x)\big)- \varphi_{\lambda}' (x) \big(\theta_{0}(x) + t\varphi_{0}(x)\big) \big)=0.
 \]

Of course the scalar products $\langle \theta_z, \varphi_{0}\rangle$ and $\langle \theta_z, \theta_{0}\rangle$ in the numerators of \eqref{eq:RES2} and \eqref{eq:RES3} can be written in the same way as~\eqref{eq:Na1} and~\eqref{eq:Na2}.

 B. Starting from Theorem~\ref{Neum}, we can everywhere replace the functions $\varphi_{0}$ and $\theta_{0}$ by $\varphi_{\zeta}$ and $\theta_{\zeta}$, where $\zeta$ is an arbitrary real fixed number. Then the construction of the paper remains unchanged if the factor $z$ in formulas
 \eqref{eq:RES2} and \eqref{eq:RES3} for $\gamma_{t}(z)$ is replaced by $z-\zeta$. The simplest way to see this is to apply the results obtained above to the operator ${\mathcal A}-\zeta I$ instead of~${\mathcal A}$.

 C.~Finally, we compare resolvent formulas in the LP and LC cases. In the LP case the resolvent $R(z)$ of a self-adjoint operator $A= A_{\min}$ is given by the relation
 \begin{equation}
 ( R (z)h) (x) = \frac{1}{\{\varphi_{z}, f_{z}\}} \bigg(f_{z}(x) \int_{0}^x \varphi_{z}(y) h (y) \,{\rm d}y+
 \varphi_{z}(x) \int_x^\infty f_{z}(y) h (y)\,{\rm d}y\bigg),
 \label{eq:R-LP}\end{equation}
 where $f_{z}(x)$ is a unique (up to a constant factor) solution of equation \eqref{eq:Jy} belonging to $L^2 ({\mathbb R}_{+})$. It can be chosen in a form $f_{z}= \theta_{z}+ w(z) \varphi_{z}$, where $w(z)$ is known as the Weyl function. Substituting this expression into~\eqref{eq:R-LP}, we see that {\it formally}
 \begin{equation}
 R (z) = w(z)\langle\cdot, \varphi_{\bar{z}}\rangle \varphi_{z} + {\mathcal R} (z),
 \label{eq:R-LP1}\end{equation}
 where $ {\mathcal R} (z)$ is given by equality \eqref{eq:RR11}. This relation looks algebraically similar to~\eqref{eq:RES1}, where~$\gamma_{t}(z)$ plays the role of the Weyl function $w(z)$. Note, however, that in the LP case~$w(z)$ is determined uniquely by the condition $ \theta_{z}+ w(z) \varphi_{z} \in L^2 ({\mathbb R}_{+})$, while in the LC case
 $\gamma_{t}(z)$ depends on the choice of a self-adjoint extension of the operator $ A_{\min}$. We also emphasize that relation~\eqref{eq:R-LP1} is only formal because $\varphi_{z}$ and $\theta_{z}$ are not in $L^2 ({\mathbb R}_{+})$
 in the LP case.

 \subsection*{Acknowledgements}

 Supported by project Russian Science Foundation 17-11-01126.

\pdfbookmark[1]{References}{ref}
\LastPageEnding

\end{document}